\documentclass[12pt,reqno]{amsart}
\advance \topmargin by -\headheight
\advance \topmargin by -\headsep
\evensidemargin \oddsidemargin
\newtheorem{theorem}{Theorem}[section]
\newtheorem{lemma}[theorem]{Lemma}

\usepackage{amsmath, amsthm, amssymb, amsfonts}
 \usepackage{hyperref}
\usepackage{enumerate}
\usepackage{url}
\usepackage{dsfont}
\usepackage{mathrsfs}

\usepackage{bm}

\usepackage{xcolor}

\usepackage{fancyvrb}

\newtheorem{prop}[theorem]{Proposition}

\theoremstyle{definition}

\theoremstyle{remark}

\numberwithin{equation}{section}

\DeclareMathOperator{\Mod}{mod}
\newcommand{\mmod}[1]{\;(\Mod{ #1})}

\renewcommand{\geq}{\geqslant}
\renewcommand{\leq}{\leqslant}

\def\eps{\varepsilon}

\def \cA {\mathcal A}

\def \cD {\mathcal D}

\def \cS {\mathcal S}

\def \N {\mathbb{N}}
\def \Q {\mathbb{Q}}
\def \R {\mathbb{R}}

\def \Z {\mathbb{Z}}

\def \lcm {\mathrm{lcm}}
\def \Mat {\mathrm{Mat}}

\title{On commuting integer matrices}

\author{Jonathan Chapman \and Akshat Mudgal}
\address{Mathematics Institute, Zeeman Building, University of Warwick, Coventry CV4 7AL, United Kingdom}
\email{Jonathan.Chapman@warwick.ac.uk}
\address{Mathematics Institute, Zeeman Building, University of Warwick, Coventry CV4 7AL, United Kingdom}
\email{Akshat.Mudgal@warwick.ac.uk}

 \subjclass[2020]{11D45 (primary); 15A27, 15B36 (secondary)} 
\keywords{Commuting matrices, restricted divisor correlations}

\begin{document}

\begin{abstract}
Given $d, N \in \N$, we define $\mathfrak{C}_d(N)$ to be the number of pairs of $d\times d$ matrices $A,B$ with entries in $[-N,N] \cap \Z$ such that $AB = BA$. We prove that
 \[ N^{10} \ll \mathfrak{C}_3(N) \ll N^{10}, \]
 thus confirming a speculation of Browning--Sawin--Wang. We further establish that 
\[ \mathfrak{C}_2(N) = K(2N+1)^5 (1 + o(1)),\]
where $K>0$ is an explicit constant.  Our methods are completely elementary and rely on upper bounds of the correct order for restricted divisor correlations with high uniformity.
\end{abstract}
\maketitle

\section{Introduction}

Given a finite group $G$, a classical topic in group theory concerns counting the number of \emph{commuting pairs} in $G$, that is, pairs $(g_1, g_2) \in G \times G$ such that $g_1 g_2 = g_2 g_1$.
Following the early work of Erd\H{o}s--Tur\'{a}n \cite{ET1968}, many results have been established on estimating numbers of commuting pairs in families of groups. For example, Gustafson \cite{Gu1973} proved that any finite, non-abelian group $G$ can have at most $5|G|^2/8$ many commuting pairs. Neumann \cite{Ne1989} further characterised all finite groups $G$ which exhibit at least $\delta|G|^2$ many commuting pairs, for some fixed $\delta >0$. 

This question has also been analysed in the case when $G$ is replaced by a suitable set $\mathscr{A}$ of $d \times d$ matrices. For instance, for all $d\geqslant 2$, Feit--Fine \cite{FF1960} obtained an explicit formula and a generating function for the number of pairs of commuting $d\times d$ matrices whose entries lie in an arbitrary finite field. More recently, for all $d \geq 3$, Browning--Sawin--Wang \cite{BSW2024} investigated the setting when  $\mathscr{A} = {\rm Mat}_d(\mathbb{Z}, N)$, that is, the set of all $d \times d$ matrices with entries in $[-N,N]\cap \mathbb{Z}$. 
Soon after, the second author \cite{Mu2024} studied commuting pairs of $2\times 2$ matrices whose entries lie in an arbitrary, finite, non-empty set of real numbers. This topic has been widely-studied from an algebraic perspective as well, see for example \cite{Ba2000, MT1955}.

In this paper, we continue the inquiry initiated in \cite{BSW2024} by establishing almost optimal estimates when $d \in \{2,3\}$ and $\mathscr{A} = {\rm Mat}_d(\mathbb{Z},N)$. Thus, we define
\begin{equation*}
    \mathfrak{C}_d(N) = |\{(A,B)\in\Mat_d(\Z,N): AB=BA\}|.
\end{equation*}
Since any scalar multiple of the $d\times d$ identity matrix commutes with every other $d\times d$ matrix, one finds that 
\begin{align} 
\mathfrak{C}_d(N) 
& \geqslant 2|\mathrm{Mat}_d(\mathbb{Z}, N)|\cdot (2N+1) ( 1- O_d(N^{-1}))  \nonumber \\
& = 2(2N+1)^{d^2+1} ( 1- O_d(N^{-1})) \gg_d N^{d^2 + 1}  \label{eqn1.1} .
\end{align}

As for upper bounds, a standard application of the divisor function estimate furnishes the bound 
\begin{equation} \label{eqn1.2}
    \mathfrak{C}_2(N) \ll_{\eps} N^{5 + \eps}
\end{equation}
for all $\eps >0$, matching \eqref{eqn1.1} up to the $O_{\eps}(N^{\eps})$ factor when $d=2$. In the case when $d \geq 3$, one may combine a result of Basili \cite{Ba2000} on the dimension of the commuting variety along with some standard Schwartz--Zippel type arguments to prove that
\[ \mathfrak{C}_d(N) \ll_d N^{d^2 + d}.\] 
Replacing the latter with work of Salberger \cite{Sa2023}, we get that $\mathfrak{C}_d(N)  \ll_{d, \varepsilon} N^{d^2 + d - 1 + \varepsilon}$ for every $\varepsilon >0$. Browning--Sawin--Wang \cite{BSW2024} further strengthened this by proving that 
 \begin{equation}\label{eqn1.3}
     \mathfrak{C}_d(N) \ll_d N^{d^2 + 2 - \frac{2}{d+1}}
 \end{equation} 
 for all $d \geq 3$. In particular, when $d=3$, their bound implies that 
\begin{equation*}
    \mathfrak{C}_3(N) \ll N^{10 + 1/2}.
\end{equation*} 

Browning--Sawin--Wang speculated that \eqref{eqn1.1} should be the correct order of magnitude for $\mathfrak{C}_d(N)$. Our first result confirms their prediction when $d=3$.

\begin{theorem} \label{thm1.1}
For every positive integer $N$, we have
\[ N^{10} \ll \mathfrak{C}_3(N) \ll  N^{10} . \]
\end{theorem}

We now turn to the case when $d=2$ wherein we have \eqref{eqn1.2}. Moreover, using \cite[Corollary $1.3$]{Mu2024}, one can remove the $O_{\eps}(N^{\eps})$ factor and show that $\mathfrak{C}_2(N) \ll N^5$. 
Our second result strengthens these results by establishing an asymptotic formula for the number of commuting pairs of $2\times 2$ integer matrices.

\begin{theorem}\label{thm1.2}
For every positive integer $N>1$, we have
\begin{equation} \label{eqn1.4}
\mathfrak{C}_2(N) = K(2N)^5 + O(N^4\log N), 
\end{equation}
where
\[    K = \frac{10\zeta(2)}{3\zeta(3)} \approx 4.56144\ldots\,.\]
\end{theorem}

A notable aspect of our estimates is that we are able to circumvent logarithmic losses. In fact, a naive application of the circle--method heuristic, which interprets the leading term as arising from the product of `local densities', would suggest the presence of some logarithmic factors in the main term in Theorem \ref{thm1.2}.
To elaborate on this, for every prime $p$ and positive integer $n$, we define $\Mat_2(\mathbb{Z}/p^n \mathbb{Z})$ to be the set of all $2 \times 2$ matrices with entries lying in $\mathbb{Z}/p^n \mathbb{Z}$ and set
\[ \mathfrak{C}_2(\mathbb{Z}/p^n \mathbb{Z}) = |\{(A,B) \in {\rm Mat}_2(\mathbb{Z}/p^n \mathbb{Z})^2 : AB = BA  \} |.  \]
We prove the following asymptotic formula for $\mathfrak{C}_2(\mathbb{Z}/p^n \mathbb{Z})$.

\begin{theorem}\label{thm1.3}
    For any prime $p$ and positive integer $n$, one has
    \[ \frac{\mathfrak{C}_2(\mathbb{Z}/p^n \mathbb{Z}) }{p^{6n}}= \left(1+\frac{1}{p}\right)\left(1-\frac{1}{p^2}\right)^{-1}\left(1-\frac{1}{p^3}\right)\left(1-\frac{1}{p^{2\left\lceil n/2\right\rceil}}\right) + O(n^2p^{-n/2}). \]
    In particular, the limit
    \begin{equation*}
        \sigma_p = \lim_{n\to\infty} \frac{\mathfrak{C}_2(\mathbb{Z}/p^n \mathbb{Z}) }{p^{6n}} = \left(1+\frac{1}{p}\right)\left(1-\frac{1}{p^2}\right)^{-1}\left(1-\frac{1}{p^3}\right)
    \end{equation*}
    exists.
\end{theorem}

The divergent product $\prod_p(1+p^{-1})$ appearing in $\prod_p \sigma_p$ suggests a logarithmic factor should be present in the main term in \eqref{eqn1.4}. This discrepancy can be explained in various ways, for instance, one would expect the circle method heuristic to hold when one expects that the solutions contributing to the main term in \eqref{eqn1.4} are somewhat equidistributed modulo every prime $p$. This does not hold in our case since a significant contribution to the main term in \eqref{eqn1.4} arises from the case when one of the matrices is a scalar multiple of the identity matrix while the other matrix is arbitrary. In geometric terms, the failure of the circle method heuristic is a consequence of the fact that the variety defined by the matrix equation $AB=BA$ is not a complete intersection.

Given the estimates provided by Theorems \ref{thm1.1} and \ref{thm1.2} as well as the lower bound in \eqref{eqn1.1}, it is natural to wonder whether one should expect to have 
\[ \mathfrak{C}_d(N) - 2(2N+1)^{d^2+1} \gg_d N^{d^2 + 1}\] 
when $d \geq 3$ or whether there is a paucity of non-trivial solutions. A relatively straightforward computation shows that the former is indeed the case.

\begin{lemma} \label{lem1.4}
For every $d \geq 2$, one has
\[ \mathfrak{C}_d(N) \geq \bigg(2 + \frac{2}{d+1} - O_d(N^{-1})\bigg)(2N+1)^{d^2 + 1}. \]
\end{lemma}

In contrast to the work of Browning--Sawin--Wang \cite{BSW2024}, which involved a variety of matrix identities, ideas from harmonic analysis and estimates concerning exponential sums in finite fields, and the work of the second author \cite{Mu2024}, which employed ideas relating to sum-product estimates, growth in groups phenomenon and incidence geometry, our techniques are purely elementary. Perhaps somewhat naturally, an important ingredient in the proof of Theorem \ref{thm1.1} involves studying higher moments of restricted divisor estimates. To elucidate this further, we define, for every $h \in \mathbb{Z}$, the function
\[ r(h) = r_N(h) := \sum_{-N\leq a_1, \dots, a_4 \leq N} \mathds{1}_{a_1a_2 - a_3 a_4 = h}. \]
We suppress the dependence of $r$ on $N$ in the notation $r(h)$ since the choice of $N$ will always be clear depending on the context. For any $k\geqslant 2$, we set 
\[ I_k(N) = \sum_{h \in \mathbb{Z}} r(h)^k .\]
Observe that $I_k(N)$ denotes the number of integer solutions to the system of equations
\[ x_1 x_2 - x_3 x_4 = \dots = x_{4k-3}x_{4k-2} - x_{4k-1}x_{4k} \]
with $|x_i| \leq N$ for every $1 \leq i \leq 4k$. Thus, $I_k(N)$ can be interpreted as the number of $k$-tuples $(A_1, \dots, A_k) \in {\rm Mat}_2(\mathbb{Z}, N)^k$ such that $\det(A_1) = \dots = \det(A_k)$. 

As part of our proof of Theorem \ref{thm1.1}, we derive the following pointwise estimates for the function $r$ and its higher moments.

\begin{lemma} \label{lem1.5}
For any $0 < |h| \leq 2N^2$, we have
    \begin{equation} \label{eqn1.5}
    r(h) \ll N^2 \sum_{\substack{d|h,\\  1 \leq d \leq N}} \frac{1}{d} \ \ \text{and} \ \ r(0) \ll N^2 \log N. 
    \end{equation}
Moreover, for any integer $k \geq 2$, we have
\begin{equation} \label{eqn1.6}
    I_k(N) \ll_k N^{2k + 2}. 
\end{equation} 
\end{lemma}

The estimates in \eqref{eqn1.5} are sharp up to the implicit constant in the Vinogradov notation for many values of $h$, as can be seen by comparing these upper bounds with the asymptotic formulae present in \cite{Af2024, DRS1993, GG2024}. Moreover, applying H\"{o}lder's inequality, we see that
\[ I_k(N) = \sum_{|h| \leq 2N^2} r(h)^k \geq \frac{(\sum_{|h| \leq 2N^2} r(h))^k}{(2N^2+1)^{k-1}}  = \frac{(2N+1)^{4k}}{(2N^2+1)^{k-1}} \gg_k N^{2k + 2}, \]
whence \eqref{eqn1.6} is also sharp up to the implicit constant.

There is a substantial body of research on $r(h)$ along with its many variants and generalisations. For instance, when $h$ is fixed and $d\geq 2$, work of Duke--Rudnick--Sarnak \cite{DRS1993} provides an asymptotic formula for all matrices $A \in {\rm Mat}_d(N, \mathbb{Z})$ such that $\det(A) = h$. One particularly well-studied variant of $r(h)$ is the classical divisor correlation 
\[  D_X(h):= \sum_{\substack{a_1, \dots, a_4 \in \mathbb{N} \\ a_1a_2, a_3a_4 \leq X}} \mathds{1}_{a_1 a_2 - a_3 a_4 = h} = \sum_{1 \leq n \leq X} \tau(n) \tau(n+h), \]
where $\tau$ is the divisor function. Obtaining an asymptotic formula for $D_X(h)$, with strong error terms and a large range of uniformity in $h$ with respect to $X$, has been an important topic of study in analytic number theory, in part due its connections to the fourth moment of the Riemann Zeta function. For details on the rich history of this problem, see \cite{HB1979, Meu2001, Mo1994} and the references therein. 

While trivially one has
\[ r_N(h) \ll D_{N^2}(h), \]
it is known that for $|h| \ll N^{5/3}$, the bound $D_{N^2}(h) \gg c_h N^2 (\log N)^2$ holds for some $c_h \gg 1$ (see \cite{HB1979}). Thus, even the sharpest estimates for $D_{N^2}(h)$ are unfortunately not strong enough to recover \eqref{eqn1.5}. It is worth mentioning that asymptotic formulae for $r(h)$ have been proven recently in \cite{GG2024} and \cite{Af2024}, but these only hold when $|h| \ll N^{1/3}$ and $|h| \ll_{\eps} N^{2 - \eps}$ for all $\eps >0$ respectively. On the other hand, while variants of $r(h)$ that incorporate smooth weights supported on dyadic intervals seem to be more amenable to analysis, with asymptotic formulae for such objects being known to hold even when $|h| \leq 2N^2$, see \cite[Theorem 1.3]{GG2024} and the ensuing discussion therein, these exhibit slightly different behaviour compared to $r(h)$ and thus do not seem to suffice for our analysis. Although it is plausible that \eqref{eqn1.5} can be recovered via the aforementioned methods, for completeness, we provide a short, elementary proof of Lemma \ref{lem1.5} in \S3. 

We conclude this section by mentioning the case when $[-N,N]\cap \Z$ is replaced by an arbitrary additively structured set $\cA \subset \mathbb{R}$. Thus, defining the sumset $\cA + \cA$ and the doubling ${\rm K}$ as
\[ \cA + \cA = \{ a + b : a , b\in \cA\} \ \ \text{and} \ \ {\rm K} = |\cA + \cA|/|\cA|,  \]
one can consider the number $\mathfrak{C}_d(\cA)$ of pairs of commuting $d \times d$ matrices such that the entries of these matrices arise from $\cA$. In \cite{Mu2024}, it was shown, amongst other estimates, that 
\begin{equation}  \label{eqn1.7}
{\rm K}^{-1} |\cA|^{5} \ll  \mathfrak{C}_2(\cA) \ll |\cA|^5 {\rm K}^{4/3}. 
\end{equation}
Note that when $\cA = [-N,N]\cap\Z$, we have $1 \leq {\rm K} \leq 2$, and so, this recovers an upper bound of the right order in this setting. Combining the techniques involved in the proof of Theorem \ref{thm1.1} along with ideas from \cite{Mu2024, So2009}, we are able to prove a corresponding estimate when $d=3$.

\begin{theorem} \label{thm1.6}
    Let $\cA \subset \mathbb{R}$ be a finite, non-empty set such that $|\cA+ \cA|/|\cA| = {\rm K}$ for some ${\rm K} \geq 1$. Then 
    \[  {\rm K}^{-3}|\cA|^{10} \ll  \mathfrak{C}_3(\cA) \ll {\rm K}^6 |\cA|^{10} (\log (2|\cA|))^3 \]
\end{theorem}

As in \eqref{eqn1.7}, the above is optimal up to the $O({\rm K}^6 (\log (2|\cA|))^3)$ factor. Moreover, this recovers the estimates in Theorem \ref{thm1.1} up to the $O( (\log 2N)^3)$ factor. Our proof of Theorem \ref{thm1.6} essentially follows the proof of Theorem \ref{thm1.1}, with restricted divisor correlations being estimated in this setting via work of Solymosi \cite{So2009} on the sum-product conjecture; see Lemma \ref{lem6.1}.

\subsection*{Organisation} In \S\ref{sec2}, we perform an analysis of $2 \times 2$ commuting matrices, thus proving Theorem \ref{thm1.2}. In \S\ref{sec3}, we study restricted divisor correlations and provide a proof of Lemma \ref{lem1.5}. We employ these along with various elementary arguments to prove Theorem \ref{thm1.1} in \S\ref{sec4}. In \S\ref{sec5}, we compute the number of $2 \times 2$ commuting matrices over $\Z/p^n \Z$ and, consequently, prove Theorem \ref{thm1.3}. Finally, we dedicate \S\ref{sec6} to proving Lemma \ref{lem1.5} and Theorem \ref{thm1.6} and for presenting some further remarks about commuting matrices.

\subsection*{Notation} We employ Vinogradov notation, that is, we write $Y \ll_{z} X$, or equivalently $Y =O_z(X)$, to mean that $|Y| \leq C_z X$, where $C_z>0$ is some constant depending on the parameter $z$. For any non-empty set $X$ and any positive integer $k$, we write $X^k = \{(x_1, \dots, x_k) : x_1, \dots, x_k \in X\}$. Given $\lambda \in \mathbb{R}$ and $(x_1, \dots, x_k) \in \mathbb{R}^k$, we define the vector $\lambda \cdot (x_1, \dots, x_k ) = (\lambda x_1, \dots, \lambda x_k)$. We usually denote the greatest common divisor of two integers $a$ and $b$ by $\gcd(a,b)$. We may also write $(a,b)$ in place of $\gcd(a,b)$ when it is clear from the context that we are examining the greatest common divisor rather than a vector with two entries. Given a property $\mathbf{P}$, we define $\mathds{1}_\mathbf{P}$ to be $1$ if $\mathbf{P}$ holds and $0$ otherwise. For example, given variables $x$ and $y$, we have $\mathds{1}_{x=y}=1$ if and only if $x=y$.

\subsection*{Acknowledgements} We are grateful to Tim Browning for introducing us to this problem during his talk at the University of Warwick Number Theory Seminar. We thank Sam Chow, Jori Merikoski, and Simon Rydin Myerson for many helpful discussions. JC is supported by EPSRC through Joel Moreira's Frontier Research Guarantee grant, ref. \texttt{EP/Y014030/1}.

\subsection*{Rights}

For the purpose of open access, the authors have applied a Creative Commons Attribution (CC-BY) licence to any Author Accepted Manuscript version arising from this submission.

\section{Commuting pairs of \texorpdfstring{$2\times 2$}{2 x 2} matrices}\label{sec2}

The purpose of this section is to prove Theorem \ref{thm1.2} and establish an asymptotic formula for $\mathfrak{C}_2(N)$. Given $(A,B) \in {\rm Mat}_2(\mathbb{Z},N)^2$, we write
\begin{equation*}
    A =
    \begin{pmatrix}
        a_1 & a_2\\
        a_3 & a_4
    \end{pmatrix},
    \qquad B=
    \begin{pmatrix}
        b_1 & b_2\\
        b_3 & b_4
    \end{pmatrix}.
\end{equation*}
Since the trace of $AB-BA$ is always zero, the diagonal entries of $AB-BA$ are both zero if and only if at least one of them is zero. This observation demonstrates that the matrix equality $AB=BA$ can be expressed as the following system of three equations:
\begin{align} 
    a_{2}b_{3} &= a_{3}b_{2}; \nonumber\\
    a_{2}(b_{4} - b_{1}) &= b_{2}(a_{4} - a_{1});\label{eqn2.1}\\
    a_{3}(b_{4} - b_{1}) &= b_{3}(a_{4} - a_{1}).\nonumber
\end{align}
One can interpret this system as the statement that the three vectors $(a_4 - a_1,b_4-b_1)$, $(a_2,b_2)$, $(a_3,b_3)$ span a subspace of $\Q^2$ (or $\R^2$) of dimension at most $1$.

We begin with some preliminary elementary lemmas on counting solutions to each individual equation appearing in the above system.

\begin{lemma}\label{lem2.1}
    Let $N\geqslant 1$, let $u,w$ be coprime, non-zero integers and let $m=\max\{|u|,|v|\}$. Then the number of solutions $(x,y)\in ([-N,N]\cap\Z)^2$ to the equation
    \begin{equation}\label{eqn2.2}
        ux = vy \neq 0
    \end{equation}
    is equal to $ 2\lfloor N/m\rfloor$.
\end{lemma}

\begin{proof}
    Since $u$ and $v$ are coprime, equation \eqref{eqn2.2} holds if and only if there exists an integer $z$ such that $(x,y)=(vz,uz)$. Moreover, we have $0< |x|,|y| \leq N$ if and only if  $0<|mz|\leqslant N$, whence, the number of choices for $(x,y)\neq(0,0)$ is $2\lfloor N/m\rfloor$, as required.
\end{proof}

\begin{lemma}\label{lem2.2}
    Let $N\geqslant 1$, let $u,v$ be coprime non-zero integers and let $m=\max\{|u|,|v|\}$. Then the number of solutions $(a_1,a_2,b_1,b_2)\in ([-N,N]\cap\Z)^4$ to the equation
    \begin{equation}\label{eqn2.3}
        u(b_1 - b_2) = v(a_1 - a_2)
    \end{equation}
    is equal to
    \begin{equation*}
        \frac{4N^3}{3m^3}(12m^2 - 6m(|u|+|v|)+ 4|uv|) + O(N^2).
    \end{equation*}
\end{lemma}

\begin{proof}
    Notice that as $u,v\neq 0$, if \eqref{eqn2.3} holds, then $a_1 = a_2$ if and only if $b_1=b_2$. Thus, there are $O(N^2)$ solutions where $a_1=a_2$ or $b_1=b_2$. It therefore suffices to count solutions which do not satisfy this condition.
    
    Since $u$ and $v$ are coprime, equation \eqref{eqn2.3} holds if and only if there exists an integer $z$ such that
\begin{equation*}
    a_{1} - a_{2} = uz, \qquad \text{and} \qquad b_{1} - b_{2} = vz.
\end{equation*}
    Our requirement that $a_1,a_2,b_1,b_2\in[-N,N]$ therefore imposes the condition $|mz|\leqslant 2N$. Thus, given $z \in \mathbb{Z}$ satisfying  $0<|mz|\leqslant 2N$, there are 
    \[ |[-N, N] \cap [-N+ uz, N + uz]\cap \mathbb{Z}| = 2N-|uz| + 1\]
    choices of $(a,a')\in([-N,N]\cap\Z)^2$ satisfying $a-a' = uz$. Similarly, there are precisely $2N-|vz|+1$ choices of $(b,b')\in ([-N,N]\cap\Z)^2$ satisfying $b-b' = |vz|$. 
    Hence, writing $M=\lfloor 2N/m\rfloor$, the total number of choices for $(a_{1},a_{2},b_{1},b_{2})$ with $a_1\neq a_2$ is
\begin{align} \label{eqn2.4}
    2\sum_{z\in[M]} & (2N - |u|z+1)(2N - |v|z +1)  = 2\sum_{z\in[M]}  (2N - |u|z)(2N - |v|z ) + O(NM) \nonumber \\
    & = 8N^2 M - 2N(|v|+|u|) M(M+1)  + \frac{1}{3}|u||v| M(M+1)(2M+1) + O(NM).
    \end{align}
Noting that $M = 2N/m + O(1)$ and 
\[ \max\{N/m,1\} (|u| + |v|) \leqslant 2N,\]
the right-hand side of \eqref{eqn2.4} simplifies to
\[ \frac{4N^3}{3m^3}\left( 12m^2 - 6m(|u|+|v|) + 4|uv|\right) + O(N^2),\]
which is the claimed estimate.
\end{proof}

Using these results, we now proceed to count `degenerate' solutions to \eqref{eqn2.1} where $a_2b_2a_3b_3 =0$. A notable finding of our analysis is that although almost all pairs $(A,B)\in \Mat_2(\Z,N)$ do not satisfy this condition, such pairs nevertheless contribute a positive proportion of commuting pairs. Indeed, as indicated in the introduction, there are $2(2N)^5 + O(N^4)$ commuting pairs $(A,B)\in {\rm Mat}_2(\mathbb{Z},N)$ with one of $A$ or $B$ equalling a scalar multiple of the identity matrix. Writing 
\[ \Gamma = \{ (A,B) \in {\rm \Mat_2}(\mathbb{Z},N)^2  : AB = BA\} \ \ \text{and} \ \ \Gamma' = \{ (A,B) \in \Gamma : a_2b_2a_3b_3 \neq 0\},\]
we now prove that almost all of the `degenerate' commuting pairs $(A,B)\in \Gamma$ must have either $A$ or $B$ being a multiple of the identity matrix.

\begin{prop}\label{prop2.3}
   We have $|\Gamma \setminus \Gamma'| = 2(2N)^5 + O(N^4\log N)$.
\end{prop}

\begin{proof}
    Our goal is to count pairs $(A,B)\in \Gamma$ with $a_{2}a_{3}b_{2}b_{3}= 0$. Notice that the first equation of \eqref{eqn2.1} ensures that at least two out of $a_2,b_2,a_3,b_3$ vanish. Suppose that at least three of these variables vanish. In view of \eqref{eqn2.1}, we deduce that either all four vanish or $0\in\{a_4-a_1, b_4-b_1\}$. In any case, we have at most four degrees of freedom in choosing the entries of $A$ and $B$, leading to at most $O(N^4)$ solutions.

    It only remains to consider $(A,B)\in\Gamma\setminus\Gamma'$ such that exactly two out of $a_2,b_2,a_3,b_3$ vanish. More precisely, the first equation of \eqref{eqn2.1} shows that exactly one element from each of the sets $\{a_2,b_3\}$ and $\{a_3,b_2\}$ must vanish. 
    If $a_2=a_3=0\neq b_2b_3$, then the second and third equations of \eqref{eqn2.1} guarantee that $a_4=a_1$, whence $A$ must be a multiple of the identity matrix. By symmetry, the condition $b_2=b_3\neq a_2a_3$ similarly implies that $B$ is a multiple of the identity matrix. Thus, these two cases correspond to pairs $(A,B)\in\Gamma\setminus\Gamma'$ where one of $A$ or $B$ is a multiple of the identity matrix and the other has non-vanishing off-diagonal entries. By inspection, there are $2(2N)^5 + O(N^4)$ such pairs.

    Finally, we have to bound the number of pairs $(A,B)\in\Gamma\setminus\Gamma'$ with $a_2=b_2=0\neq a_3b_3$ or $a_3=b_3=0\neq a_2b_2$. We focus on the latter case; the former can be handled by a symmetrical argument. 
    Our goal is therefore to bound the number of solutions to the second equation of \eqref{eqn2.1}.
 Note that fixing a choice of $(a_2,b_2)$ in our solution to \eqref{eqn2.1} with $a_2b_2 \neq 0$ is equivalent to fixing a choice of $(d,u,v) \in [N] \times\mathbb{Z}\times \mathbb{Z}$ satisfying $0<|du|,|dv|\leqslant N$; indeed, set
 $d=\gcd(a_2,b_2)$ and $(a_2,b_2)=(du,dv)$. 
    For this fixed choice of $(d,u,v)$, we infer from Lemma \ref{lem2.2} that the number of choices for $(a_1,a_4,b_1,b_4)$ is at most
    \begin{equation*}
      \ll  \frac{N^3}{\max\{|u|,|v|\}}.
    \end{equation*}
    Hence, it only remains to sum the above expression over all choices of $(d,u,v)$. By incurring an additional constant factor, we can restrict $u$ and $v$ to be positive integers with $v\leqslant u$. Thus, the number of solutions is at most
    \begin{align*}
      \ll \sum_{1 \leq d \leq N}\sum_{\substack{1 \leq v \leq u \leq N/d,\\(u,v) = 1}} \frac{N^3}{u} 
      & \leqslant \sum_{1 \leq d \leq N} \sum_{1 \leq u \leq N/d} N^3 = N^3 \sum_{1 \leq u \leq N} \sum_{1 \leq d \leq N/u} 1 \\
      & \leq N^4      \sum_{1 \leq u \leq N} \frac{1}{u} \ll N^4 \log N.
    \end{align*}
    Combining all of our estimates completes the proof.
\end{proof}

We now present our proof of Theorem \ref{thm1.2}.

\begin{proof}[Proof of Theorem \ref{thm1.2}] 
Combining Proposition \ref{prop2.3} along with the fact that
\[ \mathfrak{C}_2(N) = |\Gamma| = |\Gamma'| + |\Gamma \setminus \Gamma'|, \]
 we see that it suffices to prove that
    \begin{equation}\label{eqn2.5}
        |\Gamma'| = \left(\frac{10\zeta(2)}{3\zeta(3)} - 2\right)(2N)^5 + O(N^4\log N).
    \end{equation}
    We are therefore counting solutions to \eqref{eqn2.1} with $a_2a_3b_2b_3\neq 0$. Note that any $a_1, \dots, b_8 \in \mathbb{Z}$ with $a_2a_3b_2b_3\neq 0$ which satisfy the first two equations in \eqref{eqn2.1} must also satisfy the third equation in \eqref{eqn2.1}. As in the proof of Proposition \ref{prop2.3}, we perform a change of variables so that fixing a choice of $(a_2,b_2)$ with $a_2b_2 \neq 0$ is equivalent to fixing a choice of $(d,u,v)$ with $d=\gcd(a_2,b_2)$ and $(a_2,b_2)=(du,dv)$ for coprime $u,v\in\Z$ with $1 \leq |ud|,|vd| \leq N$. 
    Summarising, we have
    \begin{equation*} 
         |\Gamma'| = \sum_{1 \leq d \leq N }\sum_{\substack{1 \leq |u|,|v| \leq N/d, \\ (u,v) = 1}} \bigg( \sum_{0 < |a_3|,|b_3| \leq N} \mathds{1}_{b_3u = a_3v} \bigg)\bigg(  \sum_{\substack{0 \leq |a_1|, |a_4| \leq N, \\
         0 \leq |b_1|,|b_4| \leq N}} \mathds{1}_{u(b_4 - b_1) = v(a_4 - a_1)} \bigg). 
    \end{equation*}
    Writing $m = \max\{ |u|, |v|\}$, we may apply  Lemma \ref{lem2.1} and Lemma \ref{lem2.2} to deduce that
    \begin{equation}\label{eqn2.6}
         |\Gamma'| = \sum_{1 \leq d \leq N }\sum_{\substack{1 \leq |u|,|v| \leq N/d, \\ (u,v) = 1}}  \left( \frac{8N^4}{3m^4}\left( 12m^2 - 6m(|u|+|v|) + 4|uv|\right) + O\left(\frac{N^3}{m}\right) \right).
    \end{equation}
    
   As demonstrated in the proof of Proposition \ref{prop2.3}, we have the estimate 
   \[ \sum_{1 \leq d \leq N }\sum_{\substack{1 \leq |u|,|v| \leq N/d, \\ (u,v) = 1}} \frac{N^3}{m} \ll N^4 \log N.  \]
    We therefore proceed to analyse the contribution from the main term. The stipulation that $u$ and $v$ are coprime implies that $|u|=|v|$ holds if and only if $u,v\in\{-1,1\}$. The contribution of the main terms arising from this case to the right hand side of \eqref{eqn2.6} is 
    \begin{equation}\label{eqn2.7}
        4N^4\sum_{1 \leq d\leqslant N}\frac{8}{3}(12 - 6(1+1) + 4) = \frac{4}{3}(2N)^5.
    \end{equation}
   Next, note that the main term in the summand in \eqref{eqn2.6} is symmetric in $u$ and $v$ and does not depend on their signs. Thus, upon incurring a factor of $8$, we may assume that $1 \leq v < u$. Consequently, the sum of the remaining main terms in \eqref{eqn2.6} can be written as
  \begin{equation} \label{eqn2.8} 
  8 \sum_{1 \leq d\leqslant N}  \sum_{\substack{1 \leq v < u \leq N/d\\ (u,v)=1}} \frac{8N^4}{3u^4} (6u^2  - 2uv) =\frac{4}{3}(2N)^4  \sum_{1 \leq d\leqslant N}  \sum_{\substack{1 \leq v < u \leq N/d\\ (u,v)=1}} \frac{6u - 2v}{u^3} . 
  \end{equation} 
   Writing $\phi(u) = |\{ 1\leq v \leq u : \gcd(v, u) = 1\}|$ to be the Euler totient function, we see that whenever $u \geq 2$, one has
   \[    \sum_{\substack{1 \leq v < u\\ (u,v)=1}}2v = \sum_{\substack{1\leqslant v<u\\ (u,v)=1}} v + \sum_{\substack{1\leqslant v<u\\ (u,v)=1}}(u - v) = u\phi(u) . \]
        Here we have used the fact that $\gcd(u,v) = \gcd(u, u-v)$. From this identity we infer that
     \begin{align}  \label{eqn2.9}
        \frac{4}{3}(2N)^4  \sum_{1 \leq d\leqslant N}  \sum_{\substack{1 \leq v < u \leq N/d\\ (u,v)=1}} \frac{6u - 2v}{u^3}  
        & =\frac{4}{3}(2N)^4  \sum_{1 \leq d\leqslant N}\sum_{1 < u \leq N/d}  \frac{5\phi(u)}{u^2} \nonumber \\
        & =\frac{20}{3}(2N)^4\sum_{1<u\leqslant N}\frac{\phi(u)}{u^2} \sum_{1 \leq d\leqslant N/u} 1 \nonumber \\
        &= \frac{20}{3}(2N)^4\sum_{1<u\leqslant N}\left(\frac{N\phi(u)}{u^3} + O\left(\frac{\phi(u)}{u^2}\right)\right).
    \end{align}
    Putting together \eqref{eqn2.6}, \eqref{eqn2.7}, \eqref{eqn2.8}, \eqref{eqn2.9} and invoking the standard estimates
    \[ \sum_{u=2}^{N} \frac{\phi(u)}{u^3} = \frac{\zeta(2)}{\zeta(3)} - 1 + O\left(\frac{1}{N}\right) \quad \text{and} \quad  \sum_{u=2}^{N}\frac{\phi(u)}{u^2} = \frac{\log N}{\zeta(2)} +O(1)
    \]
    (see, for example, \cite[Page 71]{Ap1976}) furnishes the claimed result \eqref{eqn2.5}.
\end{proof}

\section{Restricted divisor correlations}\label{sec3}

In this section, we present a proof of Lemma \ref{lem1.5}. The key input is the following useful lemma.

    \begin{lemma} \label{lem3.1}
         For all real numbers $X \geq 1$ and positive integers $k$, we have
    \begin{equation*}
        \sum_{n\leqslant X}\left(\sum_{d\mid n}\frac{1}{d}\right)^k \ll_k X.
    \end{equation*}
    \end{lemma}
    \begin{proof}
  Expanding and changing the order of summation of the left hand side above reveals that
    \begin{align*}
        \sum_{n\leqslant X}\left(\sum_{d\mid n}\frac{1}{d}\right)^k &= \sum_{d_1,\ldots,d_k\leqslant X} \frac{|\{n\leqslant X: \lcm(d_1,\ldots,d_k)\mid n\}|}{d_1\cdots d_k} \\
        &= \sum_{d_1,\ldots,d_k\leqslant X} \frac{X}{d_1\cdots d_k\lcm(d_1,\ldots,d_k)} + O_k(\log^k X).
    \end{align*}
    By incurring a factor of $k$ in our upper bound, we can restrict the final sum to tuples $(d_1,\ldots,d_k)$ with $d_1=\max\{d_1,\ldots,d_k\}$. Since $d_1\leqslant\lcm(d_1,\ldots,d_k)$, we therefore have
    \begin{align*}
        \sum_{n\leqslant X}\left(\sum_{d\mid n}\frac{1}{d}\right)^k 
        &\ll_k\sum_{d_1\leqslant X}\frac{X}{d_1^2}\left(\sum_{d_2,\ldots,d_k\leqslant d_1}\frac{1}{d_2\cdots d_k}\right) \\
        & = \sum_{d_1\leqslant X}\frac{X}{d_1^2}((\log d_1)^{k-1} + O_k((\log d_1)^{k-2}))\ll_k X,
    \end{align*}
    as required.
\end{proof}

We now proceed to the proof of Lemma \ref{lem1.5}.

\begin{proof}[Proof of Lemma \ref{lem1.5}]
     We begin by proving the second inequality in \eqref{eqn1.5}. Observe that there are $O(N^2)$ solutions $(x_1,x_2,x_3,x_4)\in([-N,N]\cap\Z)^4$ to the equation $x_1x_2=x_3x_4$ which satisfy $x_1x_2x_3x_4=0$ or $|x_1|=|x_3|$. Hence, by writing $x_1=du$ and $x_3=dv$ with $d=\gcd(x_1,x_3)$, 
    Lemma \ref{lem2.1} implies that
    \begin{align*}
        r(0) &= O(N^2) + 2\sum_{d\leqslant N}\sum_{ \substack{0<|u|,|v|\leqslant N/d\\ (u,v) = 1}}\left\lfloor\frac{N}{\max\{|u|,|v|\}}\right\rfloor = O(N^2) + 16\sum_{d\leqslant N}\sum_{u\leqslant N/d}\sum_{\substack{v\leqslant u \\ (u,v)=1} }\frac{N}{u} \\
        &= O(N^2) + 16N^2\sum_{u\leqslant N}\frac{\phi(u)}{u^2} = \frac{16}{\zeta(2)}N^2\log N + O(N^2),
    \end{align*}
     as required.

We now prove the first inequality in \eqref{eqn1.5}. Let $h$ be an integer satisfying $0 \neq |h| \leq 2N^2$ and let $-N \leqslant a_1, \dots, a_4 \leqslant N$ satisfy $a_1 a_2 - a_3a_4 = h$. Since $h \neq 0$, the standard divisor bound shows that the contribution from solutions satisfying $a_1\dots a_4 = 0$ is $O_\eps(N^{1+\eps})$ for any $\eps>0$. Hence, we may assume that $a_1, \dots, a_4 \in \{-N, \dots, N\} \setminus \{0\}$. For any such solution, the greatest common divisor $(a_1, a_3)$ of $a_1$ and $a_3$ divides $h$. We may assume that $a_1, a_3 > 0$ since the analysis in all the other cases follows in a similar fashion. Assuming $a_1, a_3>0$, we note, as before, that the divisor function estimate implies that the number of solutions where $a_1 = a_3$ is at most $O_\eps(N^{1+\eps})$. Thus, we now analyse the contribution of the solutions which have $a_1 > a_3$. In this case, if we fix $a_1, a_3$, then the number of valid possibilities for $a_2, a_4$ is at most $\frac{(2N+1) (a_1, a_3)}{a_1} + O(1)$; indeed 
    \[ a_4 \equiv \bigg( \frac{a_3}{(a_1,a_3)} \bigg)^{-1} \frac{h}{(a_1, a_3)} \;\bigg({\rm mod}  \ 
 \frac{a_1}{(a_1, a_3)} \bigg) \quad \text{and} \quad -N \leq a_4 \leq N . \]
 Similarly, if we fix $a_1, a_3$ such that $a_1 < a_3$, then the number of valid possibilities for $a_2, a_4$ is at most $\frac{(2N+1)(a_1,a_3)}{a_3} + O(1)$. 

 Summarising the above paragraph, we have
 \[ r(h) \ll N^{1+1/2} + \sum_{\substack{0 < a_3< a_1 \leq N}} \mathds{1}_{(a_1, a_3)|h} \bigg( \frac{N (a_1, a_3)}{a_1} + O(1)\bigg).  \]
 Setting $d= (a_1, a_3)$ and noting that for fixed $a_1$ and $d$ there are at most $a_1/d$ valid choices for $a_3$, we see that
 \[ r(h) \ll N^2  + N \sum_{1 < a_1 \leq N } \frac{1}{a_1} \sum_{d | a_1} \mathds{1}_{d |h}\cdot d \left(\frac{a_1}{d}\right)   .\]
 Simplifying and changing the order of summation, we find that 
 \begin{align*} 
 r(h)  
 & \ll N^2  + N  \sum_{\substack{d|h,\\  1 \leq d < N}} \sum_{a_1 \leqslant N} \mathds{1}_{d|a_1}  \ll N \sum_{\substack{d|h,\\  1 \leq d \leq N}} ( N/d + 1) 
 \ll N^2 \sum_{\substack{d|h,\\  1 \leq d \leq N}} 1/d ,
 \end{align*} 
which is the desired estimate.

We conclude by deducing \eqref{eqn1.6} from \eqref{eqn1.5} and Lemma \ref{lem3.1}. Indeed, \eqref{eqn1.6} and Lemma \ref{lem3.1} imply that
\begin{align*}
    I_k(N)  & = \sum_{|h| \leq 2N^2} r(h)^k = r(0)^k + \sum_{0 < |h| \leq 2N^2} r(h)^k  \\
    & \ll_k N^{2k} (\log N)^k + \sum_{0 < |h| \leq 2N^2}  N^{2k} ( \sum_{d|h}1/d )^k \\
    & \ll_k N^{2k}(\log N)^{k} + N^{2k+2} \ll_k N^{2k+2}.
\end{align*}
This concludes the proof of Lemma \ref{lem1.5}.
\end{proof}

\section{Commuting pairs of \texorpdfstring{$3\times 3$}{3 x 3} matrices}\label{sec4}

In this section, we will present the proof of Theorem \ref{thm1.1}. Since the lower bound in Theorem \ref{thm1.1} follows from \eqref{eqn1.1}, it suffices to prove the upper bound $\mathfrak{C}_3(N) \ll N^{10}$. Let
\begin{equation} \label{eqn4.1}
A = \begin{pmatrix}
    a_1 & a_2 & a_3 \\
    a_4 & a_5 & a_6 \\
    a_7 & a_8 & a_9 
\end{pmatrix}
\ \ \text{and} \ \ 
B = \begin{pmatrix}
    b_1 & b_2 & b_3 \\
    b_4 & b_5 & b_6 \\
    b_7 & b_8 & b_9 
\end{pmatrix}
\end{equation}
be two matrices in ${\rm Mat}_2(\mathbb{Z}, N)$. We refer to $a_1,a_5,a_9$ and $b_1,b_5,b_9$ as the \emph{diagonal} entries of $A$ and $B$ respectively. All other entries of $A$ and $B$ are called \emph{off-diagonal} entries. For each $1 \leq i, j \leq 9$, define
\[ \cD_{i,j} = a_i b_j - a_j b_i = \det
\begin{pmatrix}
    a_i & a_j\\
    b_i & b_j
\end{pmatrix}
=\det
\begin{pmatrix}
    -b_j & a_j\\
    b_i & -a_i
\end{pmatrix}. \]
Note that $\cD_{j,i} = - \cD_{i,j}$ and $\cD_{i,i} = 0$ for every $1 \leq i,j \leq 9$. With this in hand, we see that
\[ AB - BA = \begin{pmatrix}
    \cD_{2,4} + \cD_{3,7} & \cD_{1,2} + \cD_{2,5} + \cD_{3,8} & \cD_{1,3} + \cD_{2,6} + \cD_{3,9} \\
    \cD_{4,1} + \cD_{5,4} + \cD_{6,7} & \cD_{4,2} + \cD_{6,8} & \cD_{4,3} + \cD_{5,6} + \cD_{6,9} \\
    \cD_{7,1} + \cD_{8,4} + \cD_{9,7} & \cD_{7,2} + \cD_{8,5} + \cD_{9,8} & \cD_{7,3}+\cD_{8,6} 
\end{pmatrix}.\]
Thus, the condition $AB=BA$ implies that all the entries in the above matrix are zero. In particular, $\mathfrak{C}_3(N)$ counts the number of $a_1, \dots, a_9, b_1, \dots, b_9 \in [-N,N]\cap \mathbb{Z}$ such that 
\begin{equation} \label{eqn4.2}
    a_2 b_4 - b_2 a_4 = a_7 b_3 - a_3 b_7 = a_6 b_8 - a_8 b_6
\end{equation}
and
\begin{align*} 
        -b_2(a_5- a_1) + a_2(b_5 - b_1)  + \cD_{3,8} & = 0 = -a_4 (b_5-b_1) + b_4 (a_5-a_1)  + \cD_{6,7} \\
    -b_3(a_9-a_1) + a_3 (b_9-b_1) + \cD_{2,6} &  = 0 = -a_7 (b_9-b_1) + b_7(a_9-a_1) + \cD_{8,4}   \\
   - b_8(a_9 - a_5) + a_8(b_9-b_5)  + \cD_{2,7}  & = 0 =   
   - a_6(b_9 - b_5) + b_6(a_9-a_5)  + \cD_{3,4} .
\end{align*}  
The latter six equations may be rewritten as
\begin{equation}  \label{eqn4.3}
M X = Y,
\end{equation}
where
\[ M =  \begin{pmatrix}
    -b_2 & a_2 & 0 & 0  \\
    b_4 & -a_4 & 0 & 0  \\
    0 & 0 & -b_3 & a_3  \\
    0 & 0 & b_7 & -a_7 \\
    b_8 & -a_8 & -b_8 & a_8 \\
    -b_6 & a_6 & b_6 & -a_6 \\ 
\end{pmatrix}, \quad X = 
\begin{bmatrix}
a_5-a_1 \\ 
b_5-b_1 \\
a_9-a_1 \\
b_9 - b_1 \\ 
\end{bmatrix},
\quad\text{and}\quad
Y = \begin{bmatrix}
    \cD_{8,3} \\
    \cD_{7,6} \\
    \cD_{6,2} \\
    \cD_{4,8} \\
    \cD_{7,2} \\
    \cD_{4,3}
\end{bmatrix}.
\]

We divide our analysis into multiple cases, depending on the rank of $M$. For $0\leq i \leq 4$, let $\cS_i$ be the set of all commuting pairs
$(A,B)\in\Mat_2(\mathbb{Z}, N)^2$ with ${\rm rank}(M) = i$. Thus,
\begin{equation} \label{eqn4.4}
    \mathfrak{C}_3(N) = \sum_{i=0}^4 |\cS_i|.
\end{equation} 
We therefore have to show that $|\cS_i| \ll N^{10}$ for every $0 \leq i \leq 4$. 

\begin{lemma} \label{lem4.1}
    We have $|\cS_0| + |\cS_1| + |\cS_3| \ll N^9 ( \log N)^3$.
\end{lemma}

\begin{proof}
Observe that $\rm{rank}(M)=0$ if and only if all the off-diagonal entries of the corresponding matrices $A$ and $B$ vanish. Hence, $\cS_0$ consists of all pairs of diagonal matrices from $\Mat_3(\Z,N)$, and so $|\cS_0|=(2N+1)^6$.

Suppose $(A,B)\in\cS_3$. As ${\rm rank}(M) = 3$, the determinant of any $4 \times 4$ sub-matrix of $M$ is $0$. In particular, writing $M'$ to be the $4 \times 4$ sub-matrix consisting of the first four rows of $M$, we have $\det(M') = \cD_{2,4} \cD_{3,7} = 0$. Noting \eqref{eqn4.2}, this gives us 
\begin{equation} \label{eqn4.5}
    \cD_{2,4} = \cD_{3,7} = \cD_{6,8} = 0
\end{equation} 
Applying Lemma \ref{lem1.5}, the number of choices for the off-diagonal entries of $(A,B)$ which satisfy \eqref{eqn4.5} is at most $r(0)^3 = O(N^6 (\log N)^3)$. With these terms fixed, we use the fact that ${\rm rank}(M) =3$ to deduce that there are at most $O(N)$ choices for the vector $X$ in the equation \eqref{eqn4.3}. For a fixed choice of $X$, there are $O(N^2)$ choices for the diagonal entries of $A$ and $B$. We therefore deduce that
\begin{equation*}
|\cS_3| \ll N^6 (\log N)^3 \cdot N^2 \cdot N = N^9 ( \log N)^3. 
\end{equation*}

Finally, suppose $(A,B)\in\cS_1$. We will show that this implies that at least $8$ terms out of $12$ of the off-diagonal entries of $A$ and $B$ are zero. In order to justify our claim, suppose that not all the elements in the first row of $M$ are zero, that is, one of $a_2, b_2$ is non-zero. Since ${\rm rank}(M) =1$, any two rows of $M$ are linearly dependent, and so all other rows are multiples of the first row. Applying this observation to the third, fourth, fifth and sixth rows of $M$, we see that $a_3, b_3, a_7, b_7, a_8,b_8,a_6,b_6$ are all zero. The case when any other row of the matrix $M$ has a non-zero entry can be analysed similarly. 

With this claim in hand, we can fix the off-diagonal entries of $A$ and $B$ in at most $O(N^4)$ many ways. Since ${\rm rank}(M) = 1$, for this fixed choices, there are at most $O(N^3)$ many choices for the vector $X$, whence, as in the previous cases, there are $O(N^5)$ many choices for the diagonal entries of $A$ and $B$. From this we conclude that $|\cS_1| \ll N^9$, as desired.
\end{proof}

We now turn to estimating $|\cS_4|$, which potentially contributes a positive proportion of commuting pairs.

\begin{lemma} \label{lem4.2}
    We have $|\cS_4| \ll N^{10}$.
\end{lemma}
\begin{proof}
Let $(A,B)\in\cS_4$. Since the off-diagonal entries of $A$ and $B$ satisfy \eqref{eqn4.2}, we may apply Lemma \ref{lem1.5} to fix these terms in at most $I_3(N) = O(N^8)$ many ways. Now, as ${\rm rank}(M) = 4$, there is at most one choice for the vector $X$ in \eqref{eqn4.3}.
Consequently, there are at most $O(N^2)$ choices for the diagonal entries of $A$ and $B$. Thus,
\[ |\cS_3| \ll N^8 \cdot N^2 = N^{10}, \]
which is the desired bound.
\end{proof}

Finally, we present our upper bound for $|\cS_2|$. As observed in the introduction, a positive proportion of commuting pairs lie in $\cS_2$.

\begin{lemma} \label{lem4.3}
    We have $|\cS_2| \ll N^{10}$.
\end{lemma}
\begin{proof}
Let $(A,B)\in\cS_2$. As in the proof of Lemma \ref{lem4.1}, we see that ${\rm rank}(M) = 2$ implies that \eqref{eqn4.5} holds. Hence, the first two rows, the third and fourth rows, and the fifth and sixth rows of $M$ are all pairs of linearly dependent vectors. If one of these pairs consists of two zero vectors, then we see from \eqref{eqn4.5} that there are $r(0)^2 = O(N^4(\log N)^2)$ choices for the remaining entries of $M$. Furthermore, since $M$ has rank $2$, as in the preceding case, we would be able to fix $a_1, a_5, a_9, b_1, b_5, b_9$ in $O(N^4)$ many ways. Consequently, the contribution of this subcase to $|\cS_2|$ is $O(N^8 (\log N)^2)$, which is better than the required upper bound.

Hence, we may now assume that each of the pairs described in the previous paragraph contains a non-zero vector. Explicitly, this means that each of the sets
\begin{align} \label{eqn4.6}
    \{(-b_2,a_2,0,0),(b_4,-a_4,0,0)\}, \ \ \{(0,0,-b_3,a_3),(0,0,b_7,-a_7)\}, \nonumber \\ \{(-b_6,a_6,b_6,-a_6),(b_8,-a_8,-b_8,a_8)\}
\end{align}
spans a space of dimension $1$. For concreteness, we assume that the first element of each of these sets is non-zero; the other cases can be addressed by a similar approach. Now, since $M$ has rank $2$, we deduce that there exist $\mu_1, \mu_2, \mu_3 \in\Q$, not all zero, such that
\begin{equation*}
   \mu_1 \cdot (-b_6,a_6,b_6,-a_6) +  \mu_2 \cdot (-b_2,a_2,0,0) + \mu_3 \cdot (0,0,-b_3,a_3) = 0.
\end{equation*}
This implies that 
\[ \mu_1 b_6 = - \mu_2 b_2, \quad \mu_1 a_6 = - \mu_2 a_2, \quad
\mu_1 b_6 = \mu_3 b_3, \quad \text{and}\quad \mu_1 a_6 =  \mu_3 a_3. \]
If $\mu_1 = 0$, then, as $(-b_2, a_2, 0, 0)$ and $(0,0,-b_3, a_3)$ are non-zero vectors, the preceding equations would imply that $\mu_2 = \mu_3 =0$, which contradicts our hypothesis that $\mu_1, \mu_2, \mu_3$ are not all zero. This shows that $\mu_1\neq 0$. We may similarly deduce that $\mu_2 \neq 0$ and $\mu_3 \neq 0$. Since $\mu_1, \mu_2, \mu_3 \neq 0$, the preceding equations imply that the vectors $(b_6, a_6)$ and $(b_3, a_3)$ are rational scalar multiples of $(b_2, a_2)$. Furthermore, since $\cD_{2,4} = \cD_{3,7} = \cD_{6,8} = 0$, the vectors $(b_4, a_4), (b_7, a_7)$ and $(b_8, a_8)$ must be rational scalar multiples of the vectors $(b_2, a_2), (b_3, a_3)$ and $(b_6, a_6)$. Hence, for every $j \in \{2,3,4,6,7,8\}$, there exists $\sigma_j \in \mathbb{Q}$ such that $(b_j, a_j ) = \sigma_j \cdot (b_2, a_2)$.  As previously mentioned, the cases when a different choice of vectors from the three sets described in \eqref{eqn4.6} are non-zero can be analysed similarly. 

The above argument shows that there exists $k\in\{2,3,4,6,7,8\}$ such that for all $j\in\{2,3,4,6,7,8\}$ the vector $(a_j,b_j)$ is a rational scalar multiple of the non-zero vector $(a_k,b_k)$. Writing $(a_k,b_k)=\gcd(a_k,b_k)(u,v)$ for some coprime integers $u,v\in[-N,N]$, this implies that for each $j \in \{2,3,4,6,7,8\}$ we have $(a_j,b_j)=\lambda_j(u,v)$ for some $\lambda_j\in\Z$. Thus, the total number of choices for the entries of $M$ is equal to the number of choices for $(u,v)$ multiplied by the number of choices for each $\lambda_j$.

We now compute the total number of choices for entries of $M$. For a fixed choice of coprime integers $u,v\in[-N,N]$, there are at most 
\[ 1+2\bigg\lfloor \frac{N}{\max\{|u|,|v|\}}\bigg\rfloor \]
choices of $\lambda_j$ for each $j \in \{2,3,4,6,7,8\}$, where the extra $1$ accounts for the choice $\lambda_j=0$, and the factor of $2$ corresponds to the choice of sign for $\lambda_j\neq 0$. By incurring an additional constant factor in our final upper bound, we may restrict attention to $(u,v)$ with $|u|\leqslant |v|$. Therefore, the total number of choices for the entries of $M$ is at most
\begin{equation*}
    \ll \sum_{\substack{0\leqslant |u| \leqslant |v|\leqslant N \\ \gcd(u,v)=1}} \frac{N^6}{v^6} \ll N^6 \sum_{\substack{0\leqslant u \leqslant v \leqslant N \\ \gcd(u,v)=1}} \frac{1}{v^6} \ll  N^6 +   N^6\sum_{1 \leq v \leq N}\frac{1}{v^5} \ll N^6.
\end{equation*}
Finally, having fixed the entries of $M$ in $O(N^6)$ many ways, as in the preceding cases, the assumption that $M$ has rank $2$ implies that there are at most $O(N^4)$ choices for $a_1, a_5, a_9, b_1, b_5, b_9$. Therefore, the contribution of this subcase to $|\cS_2|$ is at most $O(N^{10})$. This concludes the proof of Lemma \ref{lem4.3}.
\end{proof}

Substituting the bounds presented in Lemmas \ref{lem4.1}, \ref{lem4.2} and \ref{lem4.3} into \eqref{eqn4.4}, we deduce that 
\[ \mathfrak{C}_3(N) \ll N^{10}, \]
thus completing the proof of Theorem \ref{thm1.1}. 

\section{Commuting \texorpdfstring{$2\times 2$}{2 x 2} matrices over the \texorpdfstring{$p$}{p}-adic integers}\label{sec5}

The goal of this section is to prove Theorem \ref{thm1.3}. We are therefore interested in counting pairs of commuting $2\times 2$ matrices with entries in $\Z/p^n\Z$. The number of such pairs is equal to $p^{2n}$ multiplied by the number of solutions over $\Z/p^n\Z$ to the system
\begin{equation}\label{eqn5.1}
    a_2b_3 - a_3b_2 = a_2b_4 - a_4b_2 = a_3b_4 - a_4b_3 = 0.
\end{equation}
Indeed, we recover \eqref{eqn2.1} upon replacing $(a_4,b_4)$ with $(a_4-a_1,b_4-b_1)$ in the above system, and then observing that, for fixed $(a_4-a_1,b_4-b_1)$, there are $p^{2n}$ choices for $(a_1,a_4,b_1,b_4)$.

In contrast to the characteristic zero setting, we have to handle issues arising from zero divisors in $\Z/p^n\Z$. We accomplish this by examining $p$-adic valuations. For each prime $p$ and each non-zero integer $x$, let $\nu_p(x)$ denote the largest non-negative integer $m$ which satisfies $p^m\mid x$. Given $n\in\N$, we can then define the \emph{$p$-adic valuation} $\nu_p:\Z/p^n\Z\to \{0,1,\ldots,n\}$ by
\begin{equation*}
    \nu_p(y+p^n\Z) =
    \begin{cases}
        \nu_p(y), \quad &\text{if }\;p^n\nmid y\\
        n, &\text{if }\;p^n\mid y.
    \end{cases}
\end{equation*}
Since $\nu_p(x)=\nu_p(x+p^n)$ whenever $p^n\nmid x$, this function is well-defined. Moreover, for any $0 \leq m < n$, one has
    \begin{equation}\label{eqn5.2}
        |\{ a\in\Z/p^n\Z: \nu_p(a) = m\}|=p^{n-m}\left(1-\frac{1}{p}\right).
    \end{equation}

    When $\nu_p(a_2)=0$, say, we can invert $a_2$ in each equation in \eqref{eqn5.1} that it appears. This allows us to easily count solutions. To handle the more general setting where every variable is divisible by some $p^h$, we appeal to the following `lifting trick'.

\begin{lemma}\label{lem5.1}
    Let $n\in\N$ and let $p$ be a prime. Let $d,k,m\in\N$ and consider the system of equations
    \begin{equation*}
        F_1(x_1,\ldots,x_m) = \cdots = F_k(x_1,\ldots,x_m)=0,
    \end{equation*}
    where each $F_i$ is a homogeneous polynomial of degree $d$ with  coefficients in $\Z/p^n\Z$. For every non-negative integer $h$, let $S(n,h)$ denote the set of solutions $(x_1,\ldots,x_m)\in(\Z/p^n\Z)^m$ to the above system which satisfy
    \begin{equation*}
        \min\{\nu_p(x_1),\ldots,\nu_p(x_m)\} = h.
    \end{equation*}
    If $0\leqslant h<n/d$, then $|S(n,h)| =p^{(d-1)hm}|S(n-dh,0)|$.
\end{lemma}
\begin{proof}
    Consider the map $\Psi:S(n,h)\to S(n-dh,0)$ defined by
    \begin{equation*}
        \Psi(x_1,\ldots,x_m) = p^{-h}\cdot (x_1,\ldots,x_m) \mmod{p^{n-dh}}.
    \end{equation*}
    Here, for $a\in\Z/p^n\Z$ with $\nu_p(a)\geqslant h$, we have written $p^{-h}a$ for the unique $x\in\Z/p^{n-h}\Z$ such that $a=p^h x$. Since $dh< n$, this map is well-defined. Moreover, as the $F_i$ are all homogeneous of degree $d$, by multiplying every entry of an element of $S(n-dh,0)$ by $p^h$, we see that $\Psi$ is surjective. Finally, by viewing $p^{-h}(\mathbf{x} - \mathbf{y})$ as an element of $(\Z/p^{n-h}\Z)^m$, we find that $\Psi(\mathbf{x})=\Psi(\mathbf{y})$ if and only if $x_i \equiv y_i \mmod{p^{n - dh+h}}$ for all $1 \leq i \leq m$, that is, $\nu_p(x_i-y_i)\geqslant n-(d-1)h$ for all $1 \leq i \leq m$. Thus, for any fixed $\mathbf{y} \in S(n,h)$, there are precisely $p^{(d-1)hm}$ choices of $\mathbf{x} \in S(n,h)$ such that $\Psi(\mathbf{x}) = \Psi(\mathbf{y})$. Combining all of these observations finishes the proof.
\end{proof}

As in the integer setting, before counting the total number of solutions to \eqref{eqn5.1}, we first isolate certain `degenerate' solutions.

\begin{lemma}\label{lem5.2}
    Let $p$ be a prime. Let $n\in\N$ and set $q=p^n$. The number of solutions over $\Z/q\Z$ to the system \eqref{eqn5.1} which satisfy
    \begin{equation*}
        a_2b_3 = a_3b_2 = 0
    \end{equation*}
    is $O(n^2q^{7/2})$.
\end{lemma}
\begin{proof}
By incurring a factor of $4$ in our final bound, we may restrict attention to solutions where
    \begin{equation*}
        \nu_p(a_2)=\min\{\nu_p(a_2),\nu_p(a_3),\nu_p(b_2),\nu_p(b_3)\}
    \end{equation*}
    Writing $h = \nu_p(a_2)$, if $h=n$, then there are $q^2$ solutions to \eqref{eqn5.1} since all variables other than $a_4,b_4$ must vanish. We may therefore assume throughout that $h<n$.
    For a fixed choice of $h<n$, the number of choices for $a_2$ is $p^{n-h}(1-p^{-1})$.
    The condition that $a_2b_3 = 0$ means that $\nu_p(b_3)\geqslant n-h$, and so there are $p^h$ choices for $b_3$. Hence, there are $p^n(1-p^{-1})$ choices for the pair $(a_2,b_3)$.

    Next we count the number of choices for $(a_3,b_2)$. That is, we count the number of solutions to $a_3b_2=0$ subject to the constraint that $\nu_p(a_3),\nu_p(b_2)\geqslant h$. Notice that the equation $a_3b_2=0$ implies that $\nu_p(a_3)\geqslant n - \nu_p(b_2)$. Invoking \eqref{eqn5.2}, for $h<n/2$, we find that the number of choices for $(a_3,b_2)$ is at most
    \begin{align*}
        p^{n-h}\left(1-\frac{1}{p}\right)p^h + p^{n-(h+1)}\left(1-\frac{1}{p}\right)p^{h+1} + \cdots + p^{h}\left(1-\frac{1}{p}\right)p^{n-h} + p^{h-1}\cdot p^{n-h} \\
        = q(n-2h+1)\left(1-\frac{1}{p}\right) + \frac{q}{p} = q\left( n-2h+1 - \frac{n-2h}{p}\right).
    \end{align*}
    Here, the $(n-2h+1)$ terms in the first line come from summing over all choices of $\nu_p(b_2)\in\{h,\ldots,n-h\}$ and the final term is the number of $(a_3,b_2)$ with $p^{n-h+1}\mid b_2$ and $p^h\mid a_3$.
    If $h\geqslant n/2$, then we instead trivially bound the number of $(a_3,b_2)$ by the number of pairs of elements which are divisible by $p^h$, resulting in an upper bound of $p^{-2h}q^2$.
    
    Finally, we estimate the number of valid choices of $(a_4,b_4)$. With $(a_2,b_2)$ fixed from before, if $h=0$, then we may invert $a_2$ to show that that there are $q$ choices for $(a_4,b_4)$. When $1\leq h \leq n$, notice that $a_2b_4 = a_4b_2$ if and only if 
    \begin{equation*}
        a_2'b_4 \equiv a_4b_2' \mmod{p^{n-h}},
    \end{equation*}
    where $a_2',b_2'\in\Z/p^{n-h}\Z$ are uniquely defined by the equations $a_2=p^ha_2'$ and $b_2=p^hb_2'$. As before, there are $p^{n-h}$ solutions $a_4,b_4\in\Z/p^{n-h}\Z$ to this latter congruence. Since each $x\in\Z/p^{n-h}\Z$ lifts to exactly $p^h$ elements of $\Z/q\Z$, we deduce that,
    for any fixed choice of $(h;a_2,b_2)$, there are $p^hq$ solutions $a_4,b_4\in\Z/q\Z$ to the equation $a_2b_4 = a_4b_2$. This gives an upper bound for the number of $(a_4,b_4)$ which satisfy \eqref{eqn5.1}.
    
    Combining all of the observations from the previous three paragraphs, we conclude that the total number of solutions with $h<n$ is
    \begin{align*}
        &\ll  q^3\left(1-\frac{1}{p}\right)\sum_{0\leqslant h < n/2}p^h\left( n-2h+1 - \frac{n-2h}{p}\right) + q^4\left(1-\frac{1}{p}\right)\sum_{n/2\leqslant h < n}p^{-h}\\
        &\ll q^{7/2
        }\left(1-\frac{1}{p}\right)\sum_{0\leqslant h < n/2}\left(n+1\right) + nq^{7/2}\left(1-\frac{1}{p}\right)\ll n^2q^{7/2}.
    \end{align*}
    Here, we have also appealed to the trivial bound $p\geqslant 2$.
    Adding the $q^2<q^{7/2}$ solutions with $h=n$ to this total finishes the proof.
\end{proof}

We now use Lemma \ref{lem5.1} to count the remaining solutions to \eqref{eqn5.1} and thereby prove Theorem \ref{thm1.3}.

\begin{proof}[Proof of Theorem \ref{thm1.3}]
    
    For each $0\leqslant h<n/2$, let $S(n,h)$ denote the set of solutions to \eqref{eqn5.1} over $\Z/p^n\Z$ which satisfy
    \begin{equation*}
        h=\min\{\nu_p(a_2), \nu_p(a_3), \nu_p(a_4), \nu_p(b_2), \nu_p(b_3), \nu_p(b_4)\}.
    \end{equation*}
    We claim that
    \begin{equation*}
       |S(n,h)|= p^{4n-2h}\left(1+\frac{1}{p}\right) \left(1-\frac{1}{p^3}\right).
    \end{equation*}
    Once this is proved, Theorem \ref{thm1.3} follows by summing over all $0\leqslant h < n/2$ and using Lemma \ref{lem5.2} to bound the remaining solutions.

    Lemma \ref{lem5.1} shows that we only have to verify the above claim for $h=0$.
    For clarity, we employ the temporary notation
    \begin{equation*}
        g(a,b) := \min\{\nu_p(a),\nu_p(b)\}.
    \end{equation*}
    By symmetry of the variables, the inclusion-exclusion principle informs us that
    \begin{equation}\label{eqn5.3}
        |S(n,0)| = 3|U_0| - 3|U_1| + |U_2|,
    \end{equation}
    where
    \begin{align*}
        U_0 &=\{(a_2,a_3,a_4,b_2,b_3,b_4)\in S(n,0): g(a_2,b_2) = 0\},\\
        U_1 &=\{(a_2,a_3,a_4,b_2,b_3,b_4)\in S(n,0): g(a_2,b_2) = g(a_3,b_3) = 0\} \  \text{and} \\
        U_2 &=\{(a_2,a_3,a_4,b_2,b_3,b_4)\in S(n,0): g(a_2,b_2) = g(a_3,b_3) = g(a_4,b_4) = 0\}.
    \end{align*}
    
    Suppose we fix a pair $(a_2,b_2)$ with $g(a_2,b_2)=0$.
    Proceeding as in the proof of Lemma \ref{lem5.2}, by inverting whichever of $a_2$ or $b_2$ is a multiplicative unit, we find that there are $p^n$ choices for $a,b\in\Z/p^n\Z$ such that $a_2b = ab_2$. If we further stipulate that $g(a,b)=0$, then there are $p^n(1-p^{-1})$ choices for $(a,b)$. This comes from noting that $\nu_p(a_2)=0$ would imply that $\nu_p(a)=0$, and similarly $\nu_p(b)=0$ if $\nu_p(b_2)=0$. Moreover, given $a_2,a,a',b_2,b, b'\in\Z/p^n\Z$ with $g(a_2,b_2) = g(a,b) = g(a',b') =0$, we claim that whenever $a_2 b = b_2 a$ and $a_2 b' = b_2 a'$ hold, then $ab' = a'b$ must also hold. Indeed, if $\nu_p(a_2) = 0$, then $\nu_p(a) = \nu_p(a') = 0$, whence, $b' = a_2^{-1} b_2 a' = a^{-1} b a'$, and the case when $\nu_p(b_2) = 0$ can be similarly analysed.

    Combining all of the observations of the previous paragraph, and
    recalling that
    \begin{equation*}
        |\{(a,b)\in(\Z/p^n\Z)^2: g(a,b)=0\}| = p^{2n}\left(2 - \frac{2}{p} - \left(1 - \frac{1}{p} \right)^2\right) = p^{2n}\left(1-\frac{1}{p^2}\right),
    \end{equation*}
    we conclude that
    \begin{equation*}
        |U_k| = p^{4n}\left(1-\frac{1}{p^2}\right)\left(1-\frac{1}{p}\right)^{k} 
    \end{equation*}
    for every $0 \leq k \leq 2$. Inserting these estimates into \eqref{eqn5.3}, we obtain the claimed estimate.
\end{proof}

\section{Supplementary results and further remarks}\label{sec6}

Our aim in this section is to prove Lemma \ref{lem1.4} and Theorem \ref{thm1.6} as well as record some further remarks on estimating the number of commuting pairs of $d\times d$ matrices for $d\geqslant 4$. 

\subsection{Proof of Lemma \ref{lem1.4}}

We begin by proving Lemma \ref{lem1.4}. 
Let $A = (a_{i,j})_{1 \leq i, j \leq d}$ and $B = (b_{i,j})_{1 \leq i, j \leq d}$ be elements of ${\rm Mat}_d(\mathbb{Z}, N)$ and let $I_d$ be the $d \times d$ identity matrix. The contribution to $\mathfrak{C}_d(N)$ from the cases when either $A$ or $B$ is a multiple of the identity matrix is presented in \eqref{eqn1.1}. It therefore suffices to show that there exist at least 
    \[ \frac{2}{d+1}(2N)^{d^2 + 1}(1- O_d(N^{-1})) \]
    pairs $(A,B) \in {\rm Mat}_d(\mathbb{Z},N)$ such that
the off-diagonal entries of $A$ and $B$ are non-zero and that $AB = BA$. 
In this endeavour, note that for any fixed $A$ and any $\lambda \in \mathbb{Z}$, the matrix $B = A + \lambda  I$ commutes with $A$. In particular, any $A, B \in {\rm Mat}_d(\mathbb{Z}, N)$ satisfying
\begin{equation} \label{eqn6.1}
    a_{1,1} - b_{1,1} = a_{2,2} - b_{2,2} = \dots = a_{d,d} - b_{d,d}
\end{equation} 
and $a_{i,j} = b_{i,j}$ for every $1 \leq i, j \leq d$ with $i \neq j$ must commute. For every $i \neq j$, we can fix $a_{i,j} \in \{-N, -N+1, \dots, N\}\setminus \{0\}$ in $(2N)^{d^2 - d}$ ways. Thus, we only need to show that the number of solutions $E_d(N)$ to \eqref{eqn6.1} satisfies
\begin{equation}  \label{eqn6.2}
E_d(N) \geq \frac{2}{d+1} (2N)^{d+1}(1 - O_d(N^{-1})). 
\end{equation}

For any $n \in [-2N, 2N] \cap \mathbb{Z}$, let $r(n)$ count all pairs $n_1, n_2 \in \{-N, \dots, N\} \cap \mathbb{Z}$ such that $n = n_2 - n_1$. Since $r(n) = 2N - |n| + O(1)$, we have 
\begin{align*}
    E_d(N)
    & = \sum_{-2N \leq x \leq 2N} r(x)^d =  2\sum_{1 \leq n \leq 2N} (2N - n)^d  + O(N^d) = 2 \int_{1}^{2N} (2N  - x)^d dx + O(N^d) \\
    & = 2 (2N)^{d+1} \int_{1/2N}^{1} (1 - y)^d dy + O(N^d) 
     = \frac{2 (2N)^{d+1}}{d+1} (1- (2N)^{-1})^{d+1} + O(N^d) ,
\end{align*}
which simplifies to deliver the estimate claimed in \eqref{eqn6.2}.

\subsection{Commuting pairs over sets with small doubling}

We will now turn to proving Theorem \ref{thm1.6}, and so, for every finite, non-empty set $\cA \subset \R$ and every $n \in \mathbb{Z}$, we define
\[ r_{\cA}(n) = \sum_{a_1, \dots, a_4 \in \cA} \mathds{1}_{a_1 a_2 - a_3 a_4 =n} \ \ \text{and} \ \ I_3(\cA) = \sum_{n \in \mathbb{Z}} r_{\cA}(n)^3.  \]
With this in hand, we record the following lemma.

\begin{lemma} \label{lem6.1}
Let $\cA \subset \R$ be a finite, non-empty set with $|\cA+ \cA| = {\rm K}|\cA|$ for some ${\rm K} \geq 1$. Then 
\[ \sup_{n \in \mathbb{Z}} r_{\cA}(n) \leq r_{\cA}(0) \ll {\rm K}^2 |\cA|^2 \log (2|\cA|) \ \ \text{and} \ \ I_3(\cA) \ll {\rm K}^4 |\cA|^8 (\log (2|\cA|))^2. \]
\end{lemma}

\begin{proof}
    The second inequality follows from combining the first inequality and the fact that $\sum_{n \in \mathbb{Z}} r_{\cA}(n) = |\cA|^4$, whence it suffices to prove the first inequality. Given $n \in \mathbb{Z}$, we may apply the Cauchy--Schwarz inequality to deduce that
    \begin{align*} 
    r_{\cA}(n) & = \sum_{h \in \mathbb{Z}} \left(\sum_{a_1,a_2 \in \cA} \mathds{1}_{h +n= a_1a_2}\right)\left(\sum_{a_3,a_4 \in \cA} \mathds{1}_{h = a_3a_4}\right) \\
  &   \leq \left(\sum_{h \in \mathbb{Z}} \left(\sum_{a_1,a_2 \in \cA} \mathds{1}_{h +n= a_1a_2}\right)^2 \right)^{1/2} \left(\sum_{h \in \mathbb{Z}} \left(\sum_{a_3,a_4 \in \cA} \mathds{1}_{h = a_3a_4}\right)^2 \right)^{1/2} = r_{\cA}(0).
    \end{align*}
    Finally, employing work of Solymosi \cite{So2009} on the sum-product phenomenon, we see that 
    \[ r_{\cA}(0) \ll |\cA|^2 + \sum_{a_1, \dots, a_4 \in \cA \setminus \{0\}} \mathds{1}_{a_1a_2 = a_3 a_4} \ll {\rm K}^2 |\cA|^2 \log (2|\cA|). \qedhere\]
\end{proof}

Our aim now is to prove Theorem \ref{thm1.6}, and in this endeavour, we will follow the proof of Theorem \ref{thm1.1} and simply replace the application of Lemma \ref{lem1.5} with Lemma \ref{lem6.1}.

\begin{proof}[Proof of Theorem \ref{thm1.6}]
Given matrices $A,B$ as in \eqref{eqn4.1} with $a_1, \dots, b_9 \in \cA$, we see that $AB = BA$ if and only if \eqref{eqn4.3} and \eqref{eqn4.2} hold true. As before, we split our proof into various cases depending on ${\rm rank}(M)$. If ${\rm rank}(M) =4$, we can use Lemma \ref{lem6.1} to fix the off-diagonal entries of $A$ and $B$ in $O(I_3(\cA)) = O( |\cA|^8{\rm K}^4 (\log (2|\cA|))^2 )$ ways and then use the fact that ${\rm rank}(M) = 4$ to fix the diagonal entries of $A$ and $B$ in $O(|\cA|^2)$ many ways. Thus, when ${\rm rank}(M) =4$, we are done. When $2 \leq {\rm rank}(M) \leq 3$, the off-diagonal entries of $A$ and $B$ satisfy \eqref{eqn4.5}, and so, Lemma \ref{lem6.1} implies that these can be fixed in $O(r_{\cA}(0)^3) = O(|\cA|^6 {\rm K}^6 (\log (2|\cA|))^3$ many ways. Moreover, since ${\rm rank}(M) \geq 2$, we can fix the diagonal entries of $A$ and $B$ in $O(|\cA|^4)$ many ways, and so, we are done in this case as well. Finally, if ${\rm rank}(M) \leq 1$, then as before, we can deduce that at least $8$ of the off-diagonal entries of $A$ and $B$ are zero, in which case, we can fix the rest of entries of $A$ and $B$ in $O(|\cA|^{10})$ many ways, and so, we are done in this case as well.
\end{proof}

\subsection{Further remarks}

We conclude with a brief discussion on commuting pairs of $d\times d$ matrices with $d\geqslant 4$. We are therefore studying solutions $(A,B)\in\Mat_d(\Z,N)$ to $AB=BA$. For $d\geqslant 4$, there are additional constraints placed upon the entries of $A$ and $B$ beyond those which emerge from the approach we used when $d\leqslant 3$. Such constraints have been observed before in work on ``commuting patterns'' of matrices \cite{JM2005,JW2012}, where one seeks solutions $(A,B)\in \Mat_d(\R)$ to $AB=BA$ subject to conditions on which entries of $A$ and $B$ are allowed to be zero. For example, if we require all entries of $B$ to be non-zero, \cite[Theorem 2.4]{JM2005} exhibits certain algebraic ``swath conditions'' that the positions of the non-zero entries of $A$ must satisfy. We refer the reader to these papers for further details and references.

To illustrate the difficulties that arise for higher dimensions, we will now present a naive generalisation of our approach from \S\ref{sec4} to the case when $d \geq 4$. We split our analysis into various cases depending on the behaviour of the $d-1$ equations arising from the diagonal entries of $AB-BA$, consequently fixing the off-diagonal entries of $A$ and $B$ in different ways depending on each case. Following \S\ref{sec4}, we then write the remaining $d(d-1)$ equations as $MX=Y$, where $M$ is a $d(d-1) \times 2(d-1)$ matrix whose entries come from the off-diagonal entries of $A,B$, and $X \in \R^{2(d-1)}$ and $Y \in \R^{d(d-1)}$ are defined suitably.

Suppose that $M$ has full rank, which was one of the most significant cases studied in \S\ref{sec4}. Applying variants of Lemma \ref{lem1.5}, one may fix the off-diagonal entries of $A,B$ in $O(N^{2d^2 - 4d +2})$ many ways via the $d-1$ quadratic equations. By inverting a sub-matrix of $M$ to fix $X$, one can then fix the diagonal entries of $A,B$ in $O(N^2)$ many ways. Hence, this argument already gives an upper bound of the shape $O(N^{2d^2 - 4d+4})$. For $d\geqslant 4$, this is much larger than both the conjectured upper bound $O(N^{d^2 + 1})$ and Browning--Sawin--Wang's bound \eqref{eqn1.3}. 

One would therefore require a more refined approach which further exploits the many extra constraints present in the $d \geq 4$ setting.
In order to elucidate upon the latter, note that in the $3 \times 3$ case, one can show, either directly or via Cramer's rule, that if $X\in\R^4$ is such that the first four equations of $MX=Y$ are satisfied, then all six are satisfied. In other words, the last two equations are `consistent' with the first four. However, when $d\geqslant 4$ this property may not hold. We illustrate this with the following example. Let $d=4$ and choose
\begin{equation*}
    A = \begin{pmatrix}
a_1 & 0 & 0 & 0 \\
1 & a_6 & 0 & 0 \\
1 & a_{10} & a_{11} & 3 \\
1 & 1 & 2 & a_{16}
\end{pmatrix}
\ \ \text{and} \ \ B=
\begin{pmatrix}
b_1 & 1 & 1 & -2 \\
0 & b_6 & 0 & 0 \\
0 & b_{10} & b_{11} & 2 \\
0 & 0 & 1 & b_{16}
\end{pmatrix}.
\end{equation*}
Our choices ensure that all the diagonal entries of $AB-BA$ vanish.
One can also verify that the matrix $M$ has full rank by noting that the the first six rows form a $6\times 6$ matrix with determinant $-2$. However, the seventh row of $M$ is the zero vector whilst the seventh entry of $Y$ is $1$. Consequently, there are no solutions $X\in \R^{6}$ to the full system $MX=Y$, despite the fact that there is always a unique solution to the first six equations. Thus, the seventh equation can be `inconsistent' with the first six.

\end{document}